\documentclass[11pt]{article}
\usepackage[letterpaper,margin=1in]{geometry}
\usepackage{amsmath,amssymb,amsthm,mathtools}
\usepackage{booktabs,tabularx,array}
\usepackage{braket}
\usepackage{algorithm,algpseudocode}
\usepackage{placeins}
\newcolumntype{Y}{>{\raggedright\arraybackslash}X}
\usepackage{microtype}
\usepackage{enumitem}
\usepackage{needspace}
\usepackage[hidelinks]{hyperref}
\usepackage{fancyhdr}
\usepackage{indentfirst}
\usepackage{mathtools}
\mathtoolsset{showonlyrefs=true}
\numberwithin{equation}{section}
\newtheorem{proposition}{Proposition}[section]
\newtheorem{lemma}[proposition]{Lemma}
\newtheorem{theorem}[proposition]{Theorem}
\theoremstyle{remark}

\newcommand{\R}{\mathbb R}
\newcommand{\bbS}{\mathbb S}
\newcommand{\cI}{\mathcal I}
\newcommand{\cJ}{\mathcal J}
\newcommand{\cU}{\mathcal U}
\newcommand{\cG}{\mathcal G}
\newcommand{\cO}{\mathcal O}

\newcommand{\xu}{\overline{\mathbf x}}
\newcommand{\vx}{\mathbf x}

\newcommand{\vy}{\mathbf y}

\newcommand{\vq}{\mathbf q}
\newcommand{\ve}{\mathbf e}

\newcommand{\Argmin}{\operatorname*{Arg\,min}}

\setlist{leftmargin=*,itemsep=3pt,topsep=5pt}
\allowdisplaybreaks[1]
\title{An $\cO(1/T^3)$ algorithm for minimizing convex quadratic functions over the $L_1$ ball}
\author{Yuyuan Ouyang\thanks{School of Mathematical and Statistical Sciences, Clemson University, Clemson, SC 29634. Email: \url{yuyuano@clemson.edu}}}
\date{\today}
\begin{document}
\maketitle
\vspace{-1.5em}
\begin{abstract}
We study convex quadratic minimization over the unit $L_1$ ball in which the maximum eigenvalue of the Hessian matrix is bounded by a positive constant $L$. We propose a novel
 first-order algorithm with objective value error bounded by 
$\cO(L/T^3)$ after $T$ gradient evaluations, assuming that the subproblems involved in the algorithm can be solved exactly. To the best of our knowledge, the best convergence rate of algorithms in the literature is $\cO(L/T^2)$. 
From the perspective of information-based complexity theory, our proposed algorithm is the first in the literature that achieves the $\mathcal O((L/\varepsilon)^{1/3})$  first-order oracle complexity, although its current version is not necessarily practical for implementation. We hope that our proposed algorithm could shed some light on future implementable and efficient $\cO(L/T^3)$-convergence-rate algorithms. 
The proposed algorithm incorporates a decomposition of components of vectors in the unit $L_1$-norm ball to ``good'' and ``bad'' parts, and uses symmetric rank-1 (SR1) updates on the bad parts. 
The proposed algorithm was developed after the author instructed the OpenAI ChatGPT 6 (Astra) model to study the problem using ideas of weak-type $L^1$ estimates and good-bad part decomposition in harmonic analysis and a recent result in \cite{ouyang2026dimension}.
\end{abstract}

\section{Introduction}
\label{sec:introduction}

In this paper, our problem of interest is the following convex quadratic program:
\begin{equation}\label{eq:problem}
 f_*
 =\min_{\vx\in X}f(\vx):=\frac12\vx^\top Q\vx-\vq^\top\vx,\text{ where }X=\{\vx\in\R^n:\|\vx\|_1\le1\}\text{ and }
 0\preceq Q\preceq LI. 
\end{equation}
Here $Q\in\bbS^n$, $\vq\in\R^n$, and $L>0$.
We study first-order algorithms that solve the above problem, and analyze the efficiency of algorithms through the total number of oracle calls in order to compute an approximate solution $\xu$ such that $f(\xu) - f_*\le\varepsilon$. Our focus is on first-order methods, i.e., any oracle call should be equivalent to an inquiry on the gradient $\nabla f$. Note that we are restricting to unit ball for convenience; it is straightforward to extend our method to problems on $L_1$-norm balls of any size by rescaling the problem. 

Throughout this paper, we use the following equivalent oracle-call definition: the algorithm knows $\vq$, the oracle is a matrix-vector multiplication map $\vx\mapsto Q\vx$, and the total number of oracle calls is the total number of evaluations of such map. Our oracle call assumption is equivalent to the oracle call assumption counted by gradient evaluations, since any algorithm could always acquire the vector $\vq$ by computing $-\nabla f(0)$ and then subsequent gradient evaluations are equivalent to the matrix-vector multiplication map. We will assume a high-dimensional regime in which $n\gg T_{\max}$, where $T_{\max}$ is the largest number of oracle calls an algorithm could acquire. In addition to oracle calls, we do not make assumptions on how the algorithms would manipulate their acquired information, perform their iterations, and the arithmetic complexity of any operations the algorithms perform. For example, if we allow the $T_{\max}=n$ case that is outside of our high-dimensional regime, then under our assumption there exists a trivial algorithm for solving the problem: make inquiries $\ve^1,\ldots,\ve^{n}$ where $\ve^j$ are standard basis. With the acquired oracle outputs $Q\ve^1,\ldots,Q\ve^{n}$, the algorithm knows exactly what the matrix $Q$ is and could call an accurate solver to compute the exact solution to problem \eqref{eq:problem}. Here calling the accurate solver is acceptable, since our main focus of our desired complexity is on the information-based complexity, i.e., how many oracle calls are needed to extract enough information from $Q$ to compute an $\varepsilon$-approximate solution. The case when $T_{\max}=n$ then becomes a trivial one since the full matrix $Q$ can be recovered.

The performance limit of first-order methods on minimizing convex quadratic functions has been studied in the literature. It is shown in \cite{nemirovski1992information} that when $X$ is the unit $L_2$-norm ball, there exist worst-case problems such that the rate of convergence of any first-order method could not be better than $\cO(L/T^2)$. A tridiagonal construction for such worst case is provided in \cite{nesterov2004introductory} (see also \cite{nesterov2018lectures}), sometimes also known as ``the worst function of the world''. However, note that in such $\cO(L/T^2)$ results the norms used for objective function smoothness description (i.e. $\|\nabla f(\vx) - \nabla f(\vy)\|_2\le L\|\vx - \vy\|_2$) and constraint unit ball description are the same $L_2$-norm. In our problem of interest \eqref{eq:problem}, the norms for objective function smoothness and unit ball description are $L_2$ and $L_1$ norms respectively.
For general convex smooth problems with mismatched norms in smoothness and unit ball definition, the lower complexity bound has been studied in \cite{guzman2015lower}. In that paper, worst-case problems are developed so that the rate of convergence of any first-order method could not be better than $\cO(L/T^3)$. It should be noted that for the unit $L_2$-norm ball case, the convergence rate performance limit $\cO(L/T^2)$ is matched by the projected accelerated gradient method \cite{nesterov1983method} (see also \cite{nesterov2018lectures}). The same method could also be applied to the unit $L_1$-norm ball case with the same $\cO(L/T^2)$ convergence rate. However, to the best of the author's knowledge, the gap between the existing achievable $\cO(L/T^2)$ rate and the performance limit $\cO(L/T^3)$ has not yet been closed in the literature. Closing this gap for general convex smooth optimization has been explicitly proposed as an open problem in \cite{Guzman2015}. 

In this paper, we address the $\cO(L/T^2)$ versus $\cO(L/T^3)$ performance gap when our attention is restricted to convex quadratic programming. Specifically, we propose an algorithm that requires $T$ oracle calls to compute an approximate solution with accuracy $\cO(L/T^3)$. Our intuition is based on the good-bad decomposition in harmonic analysis. Specifically, we will decompose each iterate into a ``good'' part with controlled Euclidean norm and a ``bad'' part with controlled number of entries. By acquiring columns of $Q$ in the bad coordinates and annihilating them in the complexity analysis, the final complexity bound is evaluated on the good part with controlled Euclidean norm, which allows us to achieve an improved $\cO(L/T^3)$ convergence rate. Our main idea is described in more detail in Section \ref{sec:motivation}. The proposed algorithm and its analysis are described in Section \ref{sec:main-algorithm}. Remaining issues of the proposed algorithm and its possible resolutions, as well as further extensions of the problem, will be discussed in the concluding remarks in Section \ref{sec:conclusion}. AI use is summarized in the Acknowledgement section.

Throughout this paper, we will use the following notation conventions. We write $[n]=\{1,\ldots,n\}$ for any positive integer $n$. For any $j\in[n]$, we denote by $\ve^j$ the $j$-th
standard basis vector. For any vector $\vx\in\R^n$, we use
$x^{(j)}$ to denote the $j$-th component of $\vx$.  For
$\cG\subseteq[n]$, its complement is $\cG^c$, its cardinality is
$|\cG|$, and $\vx^\cG$ denotes the vector equal to $\vx$ on
$\cG$ and zero elsewhere.  Thus
$\vx=\vx^\cG+\vx^{\cG^c}$. For any optimization problem, we use $\Argmin$ to denote the set of all optimal solutions.

\section{Main ideas}
\label{sec:motivation}

In this section, we describe the main ideas of this paper. Specifically, we show an important property of the $L_1$ norm that allows us to decompose any components of an iterate to good and bad parts that are both controllable by a level constant. The proposed algorithm can then be designed based on the control of bad coordinates and analysis of good ones. 

\subsection{A special property of the \texorpdfstring{$L_1$}{L1}-norm}
Our algorithm is based on the following observation on a special property of the $L_1$ norm. Fix any positive real \emph{level} $\tau>0$ and consider for any $\vx\in\R^n$ the index set $\cI_\vx:=\Set{j\in [n]: |x^{(j)}|> \tau}$ and its complement $\cI^c_\vx:=\Set{j\in [n]: |x^{(j)}|\le \tau}$. 
Our observation is that 
\begin{align}
	\label{eq:l1_unique_prop}
	\begin{aligned}
		\|\vx^{\cI_\vx^c}\|_2^2 =& \sum_{j\in\cI_\vx^c} \left(x^{(j)}\right)^2 \le \tau\sum_{j\in\cI_\vx^c} |x^{(j)}|\le \tau\|\vx\|_1\text{ and }
		\\
		|\cI_\vx| = &\frac{1}{\tau}\sum_{j\in \cI_\vx} \tau \le \frac{1}{\tau}\sum_{j\in \cI_\vx} |x^{(j)}| \le \frac{1}{\tau}\|\vx\|_1. 
	\end{aligned}
\end{align}
Consequently, for any $\vx$ in the unit $L_1$-norm ball we can use the level $\tau$ to control the size of the ``good'' part $\vx^{\cI_\vx^c}$ of $\vx$ with bounded norm $\|\vx^{\cI_\vx^c}\|_2^2 \le \tau$. Moreover, the level $\tau$ also controls the total number of the ``bad'' components $|\cI_\vx|\le 1/\tau$. 
It is important to note that the choice of $\tau$ and the inequalities in \eqref{eq:l1_unique_prop} are all independent of the problem dimension $n$. A similar property can be observed for any $L_p$ norm with $p\in [1,2]$: for any $\vx$ in the unit $L_p$-norm ball, $p\in [1,2]$, we have $\|\vx^{\cI_\vx^c}\|_2^2\le \tau^{2-p}$ and $|\cI_\vx|\le \tau^{-p}$.  
Note that such properties are well known in harmonic analysis, in which many problems are addressed using this principle by controlling the magnitude of the good part and the distribution of the bad part. See, e.g., the classical Calder\'on--Zygmund lemma \cite{calderon1952existence}. 
However, note that for the $L_2$-norm we lose the control on bounding the ``good'' part using the level $\tau$, since we only have $\|\vx^{\cI_\vx^c}\|_2^2\le 1$.

The special property has great potential in algorithm design for the convex quadratic program of interest over the $L_1$-norm ball. We may demonstrate it through the following minimalist toy example. Suppose that we know in advance that in our problem of interest \eqref{eq:problem} the matrix $Q$ is block-diagonal with blocks $B\in \bbS^T_+$ and $G\in \bbS^{n-T}_+$, where $T\ll n$, and that an optimal solution $\vx_*$ satisfies $\cI_{\vx_*} \subset  [T]$ when we set the level $\tau:=1/T$. Assume that $\cJ := \cI_{\vx_*}^c$ is known. Since $\|\vx_*^{\cJ}\|_2^2\le 1/T$, we can design our algorithm as follows: first make $T$ oracle inquiries at points $\ve^j$, $j\in [T]$ to obtain the entire block $B$, then apply $T$ iterations of a proximal gradient method (starting from initial guess $\vx_0 = 0$)
\begin{align}
	\label{eq:toy}
	\begin{aligned}
		& \vx_k \in \Argmin_{\vx\in X,\ \|\vx^{\cJ}\|_2^2\le 1/T} \langle \nabla f_G(\vx_{k-1}), \vx\rangle + f_B(\vx) + \frac{L}{2}\|(\vx - \vx_{k-1})^{\cJ}\|_2^2
		\\
		&\text{ where } f_B(\vx) :=  \frac{1}{2}\vx^\top \begin{pmatrix}
			B & 0 \\ 0 & 0
			\end{pmatrix}\vx - \vq^\top \vx
		\text{ and }f_G(\vx):=f(\vx) - f_B(\vx).
	\end{aligned}
\end{align}
It is straightforward to prove that $f(\vx_T)-f_*\le \cO(L\|(\vx_0 - \vx_*)^\cJ\|_2^2/T) = \cO(L/T^2)$ after $2T$ oracle inquiries, which is already an improvement in convergence rate over the $O(1/T)$ rate of vanilla proximal gradient methods. 

Of course, the above toy example relies on extremely strong assumptions on the problem structure. It also stops at an improved $\cO(1/T^2)$ convergence rate, instead of $\cO(1/T^3)$.
In the remainder of this section, we will demonstrate how the problem structure can be explored adaptively throughout algorithm design, and how Nesterov's acceleration method (see, e.g., \cite{nesterov2018lectures}) could be incorporated to improve the convergence rate further.

\subsection{Key algorithm design ideas}

We start by observing that the block-diagonal assumption on $B$ and $G$ in the preceding toy example is unnecessary. Specifically, as long as we know the indices $\cI$ that we would like to inquire (e.g., $[T]$ in the toy example above), then we can easily build an approximation $H$ of the original Hessian $Q$ by starting from $H=0$, looping over $\cI$, and applying the symmetric rank-1 (SR1) quasi-Newton updates sequentially (see, e.g., \cite{nocedal2006numerical}). The resulting matrix $H$ would then satisfy $0\preceq H\preceq Q$ and $(Q-H)\ve^j=0$ for all $j\in\cI$. Indeed, in the SR1 method the inquiries are not only restricted to $\ve^j$. We will adopt the general SR1 implementation and focus at the $k$-th iteration of an algorithm that we are trying to design. Here we assume that we have made SR1 updates in a set of vectors and $V_k$ is a linear span of them all. Part of the vectors are standard basis ones $\{\ve^j\}_{j\in \cI_{k-1}}\subseteq V_k$ which reveal certain column information of $Q$.
Our resulting approximation is now $H_k$ that satisfies $0\preceq H_k\preceq Q$ and $(Q-H_k)\vx=0$ for all $\vx\in V_k$. We denote 
\begin{align}
	\label{eq:fHk}
	f_{H_k}(\vx):= \frac{1}{2}\vx^\top H_k\vx - \vq^\top \vx.
\end{align}
This function plays a similar role to the function $f_B(\vx)$ in the preceding toy example. 
We will describe how $\cI_k$ and $V_k$ are determined next.

In our preceding  example we assumed that $\cI_{\vx_*}$ is already known. Recall that the complement of this index set characterizes the ``good'' portion of $\vx_*$ with small $L_2$-norm, and in the example we made sure that the corresponding portions of all proximal gradient iterations also  have small $L_2$-norm. In our SR1 update description above, the corresponding index set is $\cI_{k-1}$. Observe that the function $\vx\mapsto (1/2)\vx^\top (Q-H_k) \vx$ plays a similar role to the function $f_G(\vx)$ in the previous example.
Moreover, since $(Q-H_k)\ve^j=0$ for all $j\in\cI_{k-1}$, we have for any $\vx\in X$ that
\begin{align}
	\label{eq:QHk_hessian}
	\vx^\top(Q-H_k)\vx
	=
	\left(\vx^{\cI_{k-1}^c}\right)^\top
	(Q-H_k)\vx^{\cI_{k-1}^c}
	\le
	L\|\vx^{\cI_{k-1}^c}\|_2^2.
\end{align}
Consequently, if for some level $\tau>0$ we have $
|x^{(j)}|\le\tau$ for all $j\in\cI_{k-1}^c,
$
or equivalently $\cI_\vx\subseteq\cI_{k-1}$, then \eqref{eq:l1_unique_prop} yields
$
\|\vx^{\cI_{k-1}^c}\|_2^2\le\tau
$
and hence
\[
f(\vx)-f_{H_k}(\vx)
=
\frac{1}{2}\vx^\top(Q-H_k)\vx
\le
\frac{L\tau}{2}.
\]
Thus, once $\cI_{k-1}$ contains all coordinates of $\vx$ above the level $\tau$, the learned quadratic $f_{H_k}$ approximates $f$ at $\vx$ with error at most $L\tau/2$. The question is how to construct $\cI_k$. To this end, we fix the level $\tau>0$ and associate $\cI_k$ with a sequence of iterates $\vx_k$. Specifically, given $\cI_{k-1}$, we compute $\vx_k$ and add its newly identified coordinates to $\cI_k$. Equivalently, we decide which columns of $Q$ to inquire about based on the large components of $\vx_k$. This update guarantees that $\cI_{\vx_k}\subseteq \cI_k$. By the special property of the $L_1$-norm, outside $\cI_k$ we have $\|\vx_k^{\cI_k^c}\|_2^2 \le \|\vx_k^{\cI_{\vx_k}^c}\|_2^2\le \tau$. For $k\ge1$, we use $V_k$ to denote the span of $\vx_0,\ldots,\vx_{k-1}$ and the standard basis vectors $\ve^j$ with $j\in\cI_{k-1}$, which are precisely the inquiries that build up $H_k$.

After initialization, we update $\vx_k$ in the following way:
$$
\vx_k\in\Argmin_{\vx\in X}
f_{H_k}(\vx)+L\tau\Gamma_k p_{\cI_{k-1}}(\vx),
$$
where the weight $L\tau\Gamma_k$ balances the minimization of the approximation $f_{H_k}$ to the original function $f$ and regularization $p_{\cI_{k-1}}$. Here $f_{H_k}$ is defined in \eqref{eq:fHk}, and for any $\cI\subseteq [n]$ and any $\vx\in X$ we define $p_{\cI}(\vx)$ by two summands of Huber and absolute functions:
\begin{align}
	\label{eq:p}
	p_{\cI}(\vx)
	:=
	\sum_{j\in\cI^c}\psi_\tau(x^{(j)})
	+\sum_{j\in\cI}|x^{(j)}|,\text{ where }
	\psi_\tau(s):=
	\begin{cases}
		s^2/(2\tau),& |s|\le\tau,\\
		|s|-\tau/2,& |s|>\tau.
	\end{cases}
\end{align}
Note that we do not have the $\langle\nabla f_G(\vx_{k-1}),\vx\rangle$ terms that appeared in the preceding example; this is since $\vx\mapsto (1/2)\vx^\top (Q-H_k) \vx$ plays the role of $f_G$ but its gradient $ (Q-H_k) \vx=0$ for any $\vx\in V_k$. However, the Hessian term of this function is bounded by $\|\vx^{\cI_{k-1}^c}\|_2^2$ in \eqref{eq:QHk_hessian}, corresponding to the summand over $j\in\cI_{k-1}^c$ in the definition of $p_{\cI_{k-1}}$. For the components $x^{(j)}$ that are possibly large (either $j\in\cI_{k-1}$ or $|s|> \tau$ in the definition), we replace the larger quadratic penalty by a slower growing absolute value penalty so that we do not over-penalize the ``bad'' portions of $\vx_k$. In fact, since the Huber function $\psi_\tau(x^{(j)})\le |x^{(j)}|$, we have
\begin{align}
	\label{eq:p_bound}
0\le p_{\cI_{k-1}}(\vx)\le \|\vx\|_1\le1
\qquad\text{for all }\vx\in X,
\end{align}
so the regularization penalty will not be extremely large.

\section{The proposed algorithm}
\label{sec:main-algorithm}

We now state the algorithm suggested by the discussion in the previous section. Our proposed algorithm is described in Algorithm \ref{alg:main}.

\begin{algorithm}[!htbp]
\caption{\label{alg:main} The proposed algorithm}
\begin{algorithmic}[1]
\Require Maximum number of oracle calls $T_{\max}\ge0$,
level $0<\tau\le1$, $L>0$, and $\vq$ in problem \eqref{eq:problem}.
\State Set $T\gets0$, $H_0\gets0$,
$\Gamma_0\gets1/\tau$, and
$\cI_{-1}\gets\varnothing$. Choose
\(
 \vx_0\in\Argmin_{\vx\in X}
 \left\{L\tau\Gamma_0 p_{\emptyset}(\vx)-\vq^\top\vx\right\}\), and set
$\xu_0\gets\vx_0$, 
$
 \cI_0\gets\{j:|x_0^{(j)}|>\tau\},
$
and $\cU_0 \gets \cI_0$. 
\For{$k=1,2,\ldots$}
  \State Set $\widetilde H_k\gets H_{k-1}$
  \For{$j\in\cU_{k-1}$}
    \State $(\widetilde H_k,T)
    \gets\Call{SR1}{\widetilde H_k,\ve^j,T}$
  \EndFor
  \State $(H_k,T)
  \gets\Call{SR1}{\widetilde H_k,\vx_{k-1},T}$
  \State Set
  \(
  \displaystyle
  \bar\gamma_k\gets
  {2}/\left({1+\sqrt{1+4/\Gamma_{k-1}}}\right).
  \)
  \State For $0\le\gamma\le\bar\gamma_k$, define
  \begin{align}
  	\label{eq:subproblem}
  X_k^*(\gamma)
  &:={}\Argmin_{\vx\in X}
  \left\{L\tau\Gamma_{k-1}(1-\gamma)
  p_{\cI_{k-1}}(\vx)+f_{H_k}(\vx)\right\},
  \\
      Y_k(\gamma)
&:=\left\{\vx\in X_k^*(\gamma):
|x^{(j)}|\le\tau\text{ for all }j\in\cI_{k-1}^c\right\}.
\end{align}
  \State Choose $\vx_k\in X_k^*(0)$
  \If{$|x_k^{(j)}|\ge\tau$ for some $j\in\cI_{k-1}^c$}
    \State $\gamma_k\gets0$
  \ElsIf{$Y_k(\bar\gamma_k)\ne\varnothing$}
    \State $\gamma_k\gets\bar\gamma_k$ and replace $\vx_k$ by any
    $\vx_k\in Y_k(\bar\gamma_k)$
  \Else
    \State Set $\gamma_k$ to the largest $\gamma\in(0,\bar\gamma_k)$ such that $Y_k(\gamma)$ is non-empty, and replace $\vx_k$ by any $\vx_k\in Y_k(\gamma_k)$ satisfying $|x_k^{(j)}|=\tau$ for some $j\in\cI_{k-1}^c$
  \EndIf
  \State Set
  \(
  \displaystyle
  \Gamma_k\gets\Gamma_{k-1}(1-\gamma_k),
  \) \(
  \xu_k\gets(1-\gamma_k)\xu_{k-1}+\gamma_k\vx_k.\)
  \State Set $\cI_k = \cI_{k-1}\cup \cU_k$, where 
  \[
  \cU_k\gets\begin{cases}
  \{j\in\cI_{k-1}^c:\ |x_k^{(j)}|\ge\tau\} & \text{ if }\gamma_k<\bar\gamma_k
  	\\
  	\emptyset & \text {otherwise}.
  	\end{cases}
  \]
\EndFor
\Statex
\Function{$(H_+,T_+)=$SR1}{$H, \vx,T$}
\State Set $T_+\gets T+1$ and compute
\begin{align}
	\label{eq:sr1_update}
	H_+ = \begin{cases}
		H + \left({\vx^\top (Q - H)\vx}\right)^{-1}(Q-H)\vx\vx^\top(Q-H) & \text{ if }{\vx^\top (Q - H)\vx}>0
		\\
		H\text{ otherwise}.
	\end{cases}
\end{align}
\State If $T_+\ge T_{\max}$, terminate the algorithm.
\EndFunction
\end{algorithmic}
\end{algorithm}

A few remarks on Algorithm \ref{alg:main} are in place. 
First,
the index set $\cI_k$ records the bad coordinates selected for $Q$-column inquiry, while
$\cI_k^c$ records the provisional good coordinates. A coordinate in the
good part at one iteration may be added to $\cI_k$ later.
Second,
the coordinates outside $\cI_{k-1}$ with magnitude greater than $\tau$
are bad, while those with magnitude at most $\tau$ are good. Whenever
$\gamma_k<\bar\gamma_k$, all coordinates with magnitude at least $\tau$
are added to $\cI_k$. Among them, if we also have $\gamma_k>0$, then $\vx_k\in Y_k(\gamma_k)$,
so every coordinate added at that iteration has magnitude exactly $\tau$.
Third,
the choice of $\bar\gamma_k$ satisfies
\[
 \bar\gamma_k^2
 =\Gamma_{k-1}(1-\bar\gamma_k),\quad k\ge1.
\]
This is the well-known Nesterov-type parameter update equation (see, e.g., \cite{nesterov2018lectures}). From this
perspective, our good-bad decomposition can be seen as an analogous
non-Euclidean Nesterov-type update with Lipschitz constant $L$ and
strong convexity constant $1/\tau$ from the distance generating
function when restricted to good coordinates (here the $1/\tau$
coefficient comes from the Huber function definition).
Fourth, 
Algorithm~\ref{alg:main} is written with exact subproblem
solutions because our focus of this paper is the oracle
complexity. Their replacement by approximate solvers is discussed in our concluding remarks later in 
Section~\ref{sec:conclusion}.
Fifth, 
note that there are three possibilities in solving the subproblem \eqref{eq:subproblem}. Preferably, we would like $Y_k(\bar\gamma_k)$ to be nonempty, in which case our good-bad decomposition of iterates remains the same. Otherwise, we enlarge $\cI_{k-1}$ to form $\cI_k$. This happens with either $\gamma_k=0$ or $\gamma_k\in(0,\bar\gamma_k)$. We will show later in Lemma \ref{lem:subproblem-existence} that there always exist $\gamma_k$ and $\vx_k$ that satisfy at least one of the three possibilities.
Finally, 
it can be observed from the SR1 function description that each SR1 call requires exactly one oracle call $\vx\mapsto Q\vx$. When an index $j$ enters $\cI_k$, the algorithm makes the
inquiry of the column $Q\ve^j$ 
and incorporates it through the SR1 update. The properties of the SR1 updates are well-known in the literature (see, e.g., \cite{nocedal2006numerical}); we state the properties below without proof.

\begin{proposition}
	\label{thm:sr1_prop}
	Suppose that a sequence of SR1 calls of the form $(Q_i,T)=$SR1$(Q_{i-1}, \vx_{i-1},T)$ is made at an inquiry sequence $\{\vx_i\}_{i=0}^{t-1}$ with $Q_0=0$ and $\{Q_i\}_{i=1}^t$ is the sequence of outputs. Then we have
	\begin{align}
		0\preceq Q_{i-1}\preceq Q_i\preceq Q\text{ and }(Q-Q_i)\vx = 0,\ \forall \vx\in\operatorname{span}\{\vx_0,\ldots,\vx_{i-1}\},\ \forall i=1,\ldots, t
	\end{align}
\end{proposition}

We are now ready to prove the convergence properties of the proposed algorithm. We start by showing that Algorithm \ref{alg:main} is well-defined. Specifically, in the following lemma we will show that an iterate $\vx_k$ and coefficient $\gamma_k$ that satisfy the description of the algorithm always exist. 
\begin{lemma}
	\label{lem:subproblem-existence}
	At every iteration $k\ge1$, the choices of $\gamma_k$ and $\vx_k$
	in Algorithm~\ref{alg:main} can always be made. In particular, if the point
	chosen from $X_k^*(0)$ satisfies
	\[
	 |x_k^{(j)}|<\tau,\qquad j\in\cI_{k-1}^c,
	\]
	and $Y_k(\bar\gamma_k)=\varnothing$, then
	\[
	 \gamma_k:=\max\{\gamma\in(0,\bar\gamma_k):
	 Y_k(\gamma)\ne\varnothing\}
	\]
	is well-defined, and $Y_k(\gamma_k)$ contains a point $\vx_k$
	satisfying $|x_k^{(j)}|=\tau$ for some $j\in\cI_{k-1}^c$.
\end{lemma}
\begin{proof}
	Fix $k\ge1$. The objective defining $X_k^*(\gamma)$ is continuous
	in $(\gamma,\vx)$, and $X$ is compact. Hence $X_k^*(\gamma)$ is
	nonempty and compact for every $\gamma\in[0,\bar\gamma_k]$, and
	\[
	 \{(\gamma,\vx):0\le\gamma\le\bar\gamma_k,
	 \ \vx\in X_k^*(\gamma)\}
	\]
	is compact. It follows that
	\[
	 D:=\{\gamma\in[0,\bar\gamma_k]:Y_k(\gamma)\ne\varnothing\}
	\]
	is compact.

	The point selected from $X_k^*(0)$ exists. If it satisfies
	$|x_k^{(j)}|\ge\tau$ for some $j\in\cI_{k-1}^c$, the algorithm sets
	$\gamma_k=0$. Otherwise, $|x_k^{(j)}|<\tau$ for every
	$j\in\cI_{k-1}^c$. Every point in $X_k^*(0)$ then has the same
	coordinates indexed by $\cI_{k-1}^c$. Indeed, if two minimizers
	differed in one of these coordinates, a sufficiently short part of
	the segment joining them and starting from $\vx_k$ would remain in
	the quadratic part of the Huber function. Since the coefficient of
	the Huber term is positive, the objective would be strictly convex
	along this part of the segment, whereas every point on the segment
	between two minimizers must also be a minimizer. This is a
	contradiction. Thus every point in $X_k^*(0)$ satisfies
	$|x^{(j)}|<\tau$ for all $j\in\cI_{k-1}^c$.

	We claim that $[0,\epsilon]\subseteq D$ for some $\epsilon>0$.
	Otherwise, there would be a sequence $\gamma_\ell\downarrow0$ with
	$\gamma_\ell\notin D$. Choose $\vx_\ell\in X_k^*(\gamma_\ell)$.
	After passing to a subsequence, the vectors $\vx_\ell$ converge and
	the same index $j\in\cI_{k-1}^c$ satisfies
	$|x_\ell^{(j)}|>\tau$ for every $\ell$. The limit belongs to
	$X_k^*(0)$ and satisfies $|x^{(j)}|\ge\tau$, contradicting the
	preceding conclusion. This proves the claim.

	If $Y_k(\bar\gamma_k)$ is nonempty, the algorithm sets
	$\gamma_k=\bar\gamma_k$ and chooses $\vx_k$ from that set. Suppose
	that it is empty. Compactness of $D$ and the above claim then allow us to find
	\[
	 \gamma_k:=\max D\in(0,\bar\gamma_k).
	\]
	Choose any $\vx_k\in Y_k(\gamma_k)$. If all its coordinates indexed
	by $\cI_{k-1}^c$ had magnitude strictly smaller than $\tau$, the same
	strict-convexity argument would show that every point in
	$X_k^*(\gamma_k)$ has those same coordinates. Choose
	$\gamma_\ell\downarrow\gamma_k$ with $\gamma_\ell>\gamma_k$ and
	$\vx_\ell\in X_k^*(\gamma_\ell)$. Since $\gamma_\ell\notin D$,
	each $\vx_\ell$ has a coordinate indexed by $\cI_{k-1}^c$ with
	magnitude greater than $\tau$. After passing to a subsequence, the
	vectors converge, and the same coordinate has magnitude at least
	$\tau$ at the limit. The limit belongs to $X_k^*(\gamma_k)$, a
	contradiction. Therefore $|x_k^{(j)}|=\tau$ for some
	$j\in\cI_{k-1}^c$, as required.
\end{proof}

With iterate $\vx_k$ well-defined, we can now present in the following proposition the recursion properties of the iterates in Algorithm \ref{alg:main}.

\begin{proposition}
	\label{thm:recur}
	For any iteration $k\ge 1$ at which Algorithm \ref{alg:main} does not terminate, we have
	\begin{align}
		\label{eq:recur_pfx}
		\begin{aligned}
		& L\tau\Gamma_k p_{\cI_k}(\vx_k) + f_{H_k}(\vx_k) 
		\\
		\ge & (1-\gamma_k)(L\tau\Gamma_{k-1} p_{\cI_{k-1}}(\vx_{k-1}) + f_{H_{k-1}}(\vx_{k-1})) 
		\\
		& + \gamma_k f_{H_k}(\vx_k) + \frac{L\gamma_k^2}{2}
		\|(\vx_k-\vx_{k-1})^{\cI_{k-1}^c}\|_2^2
		+\frac{L\tau^2\Gamma_k}{2}
		|\cU_{k}|
	\end{aligned}
	\end{align}
	and
	\begin{align}
		\label{eq:recur_fxu}
		 f(\xu_k)
		\le{}&(1-\gamma_k)f(\xu_{k-1})
		+\gamma_k f_{H_k}(\vx_k)+\frac{L\gamma_k^2}{2}
		\|(\vx_k-\vx_{k-1})^{\cI_{k-1}^c}\|_2^2.
	\end{align}
\end{proposition}
\begin{proof}
For every $k\ge0$, the initialization and the subproblem definition yield
\begin{align}
	\label{eq:x}
	\vx_k \in \Argmin_{\vx\in X}
	\left\{L\tau \Gamma_k p_{\cI_{k-1}}(\vx) + f_{H_k}(\vx)\right\},
	\text{ and } |x_k^{(j)}|\le \tau,\quad j\in \cI_k^c.
\end{align}
Moreover, if $k\ge1$ and $\gamma_k>0$, then $\vx_k\in Y_k(\gamma_k)$ and hence
$|x_k^{(j)}|\le\tau$ for all $j\in\cI_{k-1}^c$.
For every $k\ge0$ and $j\in\cU_k$, we also have
$|x_k^{(j)}|\ge\tau$. Therefore
\begin{align}
	\label{eq:p_const_change}
	p_{\cI_k}(\vx_k) - p_{\cI_{k-1}}(\vx_k)
	= \frac{\tau}{2}|\cU_k|.
\end{align}
Moreover, at these coordinates, the Huber function and the absolute-value
function have the same subdifferential. The optimality condition in \eqref{eq:x}
thus also gives
\begin{align}
	\label{eq:x_also}
	\vx_k \in \Argmin_{\vx\in X}
	\left\{L\tau \Gamma_k p_{\cI_{k}}(\vx) + f_{H_k}(\vx)\right\}.
\end{align}
For any $\vx\in X$ such that $|x^{(j)}|\le \tau$ for all
$j\in\cI_k^c$, the Huber terms indexed by $\cI_k^c$ are quadratic
along the segment between $\vx_k$ and $\vx$. Hence \eqref{eq:x_also}
and the optimality condition of $\vx_k$ imply
\begin{align}
	\label{eq:x_oc}
	L\tau \Gamma_k p_{\cI_{k}}(\vx) + f_{H_k}(\vx)
	-L\tau \Gamma_k p_{\cI_{k}}(\vx_k) - f_{H_k}(\vx_k)
	\ge \frac{L\Gamma_k}{2}
	\|(\vx - \vx_k)^{\cI_k^c}\|_2^2.
\end{align}

Suppose first that $\gamma_k>0$. Applying \eqref{eq:x_oc} at $k-1$
with $\vx=\vx_k$ gives
\begin{align}
	L\tau \Gamma_{k-1} p_{\cI_{k-1}}(\vx_k) + f_{H_{k-1}}(\vx_k) - L\tau \Gamma_{k-1} p_{\cI_{k-1}}(\vx_{k-1}) - f_{H_{k-1}}(\vx_{k-1}) \ge \frac{L\Gamma_{k-1}}{2}\|(\vx_k - \vx_{k-1})^{\cI_{k-1}^c}\|_2^2.
\end{align}
By \eqref{eq:p_const_change} and Proposition \ref{thm:sr1_prop},
\begin{align}
	L\tau \Gamma_{k-1} p_{\cI_{k}}(\vx_k) + f_{H_{k}}(\vx_k) \ge L\tau \Gamma_{k-1} p_{\cI_{k-1}}(\vx_k) + f_{H_{k-1}}(\vx_k) + \frac{L\tau^2\Gamma_{k-1}}{2}|\cU_k|.
\end{align}
Adding these inequalities, multiplying by $(1-\gamma_k)$, and adding
$\gamma_k f_{H_k}(\vx_k)$ gives the right-hand side of
\eqref{eq:recur_pfx} with $L\Gamma_k/2$ in place of
$L\gamma_k^2/2$. Since
\[
 \gamma_k^2\le\bar\gamma_k^2
 =\Gamma_{k-1}(1-\bar\gamma_k)
 \le\Gamma_{k-1}(1-\gamma_k)=\Gamma_k,
\]
this proves \eqref{eq:recur_pfx} when $\gamma_k>0$.

If $\gamma_k=0$, then $\Gamma_k=\Gamma_{k-1}$. Since $\vx_k$
minimizes the objective in \eqref{eq:x}, $H_k\succeq H_{k-1}$, and
$\vx_{k-1}$ satisfies \eqref{eq:x_also},
\begin{align*}
 &L\tau\Gamma_kp_{\cI_{k-1}}(\vx_k)+f_{H_k}(\vx_k)
 \\
 &\quad\ge
 L\tau\Gamma_{k-1}p_{\cI_{k-1}}(\vx_{k-1})
 +f_{H_{k-1}}(\vx_{k-1}).
\end{align*}
Combining this inequality with \eqref{eq:p_const_change} proves
\eqref{eq:recur_pfx} also when $\gamma_k=0$.

Before $\vx_k$ is selected, the SR1 calls have inquired
$\vx_0,\ldots,\vx_{k-1}$. Hence Proposition \ref{thm:sr1_prop}
shows that $Q-H_k$ annihilates $\vx_0,\ldots,\vx_{k-1}$.
Since $\xu_{k-1}$ is a convex combination of these vectors,
$Q-H_k$ also annihilates $\xu_{k-1}$. Consequently,
\begin{align}
	f(\xu_k)
	= & f_{H_k}(\xu_k)+
	\frac{\gamma_k^2}{2}(\vx_k- \vx_{k-1})^\top(Q-H_k)
	(\vx_k- \vx_{k-1})
	\\
	\le & (1-\gamma_k)f_{H_k}(\xu_{k-1})+ \gamma_k f_{H_k}(\vx_k) + 
	\frac{L\gamma_k^2}{2}\|(\vx_k- \vx_{k-1})^{\cI_{k-1}^c}\|_2^2
	\\
	= & (1-\gamma_k)f(\xu_{k-1}) + \gamma_k f_{H_k}(\vx_k) +
	\frac{L\gamma_k^2}{2}\|(\vx_k- \vx_{k-1})^{\cI_{k-1}^c}\|_2^2
\end{align}
where the inequality follows from convexity of $f_{H_k}$ and
\eqref{eq:QHk_hessian}. This proves \eqref{eq:recur_fxu}.
\end{proof}

With the help of the above recursions, we are ready to describe the convergence property of Algorithm \ref{alg:main} below. 

\begin{theorem}\label{thm:main}
Suppose that 
the SR1 function call terminates Algorithm \ref{alg:main} at $k=K+1$. The approximate solution $\xu_K$ produced by the algorithm satisfies
\begin{equation}
	f(\xu_K)-f_*
\le L\tau\Gamma_K\left(1-\frac{\tau}{2}|\cI_K|\right),\text{ where }|\cI_K|\le \frac{2}{\tau}.
	\end{equation}
\end{theorem}
\begin{proof}
By \eqref{eq:recur_pfx} and \eqref{eq:recur_fxu} in Proposition \ref{thm:recur} we have
	\begin{align}
	\begin{aligned}
		& L\tau\Gamma_k p_{\cI_k}(\vx_k) + f_{H_k}(\vx_k) - f(\xu_k)
		\\
		\ge & (1-\gamma_k)(L\tau\Gamma_{k-1} p_{\cI_{k-1}}(\vx_{k-1}) + f_{H_{k-1}}(\vx_{k-1}) - f(\xu_{k-1}))  +\frac{L\tau^2\Gamma_k}{2}
		|\cU_{k}|.
	\end{aligned}
\end{align}
Noting that $\Gamma_k = \Gamma_{k-1}(1-\gamma_k)$, dividing the above relation by $\Gamma_k$  and summing up we have
\begin{align}
	& \frac{1}{\Gamma_k}\left(L\tau\Gamma_k p_{\cI_k}(\vx_k) + f_{H_k}(\vx_k) - f(\xu_k)\right) 
	\\
	\ge & \frac{1-\gamma_1}{\Gamma_1}(L\tau\Gamma_{0} p_{\cI_{0}}(\vx_{0}) + f_{H_{0}}(\vx_{0}) - f(\xu_{0})) + \frac{L\tau^2}{2}\sum_{i=1}^{k}|\cU_i|
	\\
	= & \tau\left(Lp_{\emptyset}(\vx_0)
	-\frac12\vx_0^\top Q\vx_0\right)
	+\frac{L\tau^2}{2}\sum_{i=0}^{k}|\cU_i|
	\\
	\ge & \frac{L\tau^2}{2}|\cI_k|.
\end{align}
Here the last inequality follows from $Q\preceq LI$ and the fact that
$\psi_\tau(s)\ge s^2/2$ for $|s|\le1$ and $0<\tau\le1$.

Let $\vx_*$ be an optimal solution of our problem of interest \eqref{eq:problem}. By \eqref{eq:x_also}, \eqref{eq:p_bound}, and Proposition \ref{thm:sr1_prop}, the above
inequality implies that
\begin{align}
 &\frac{L\tau^2\Gamma_k}{2}|\cI_k|+f(\xu_k)
 \le L\tau\Gamma_k p_{\cI_k}(\vx_k)+f_{H_k}(\vx_k)\le L\tau\Gamma_k p_{\cI_k}(\vx_*)+f_{H_k}(\vx_*)
 \\
 \le& L\tau\Gamma_k+f(\vx_*).
\end{align}
Letting $k=K$ in the above relation we have
\begin{align}
	f(\xu_K)-f_*
	\le L\tau\Gamma_K\left(1-\frac{\tau}{2}|\cI_K|\right).
\end{align}
Since $f(\xu_K)\ge f_*$, we also conclude that $|\cI_K|\le 2/\tau$. 
\end{proof}

It suffices to bound $\Gamma_k$ in the above relation to obtain the final convergence rate, as studied in the theorem below. 

\begin{theorem}
	Suppose that 
	the maximum number of oracle calls $T_{\max}\ge 5$, the level $\tau = 6/(T_{\max}+1)$, and the SR1 function call terminates Algorithm \ref{alg:main} at $k=K+1$. The approximate solution $\xu_K$ produced by the algorithm satisfies
	\begin{align}
		\label{eq:main_rate}
		f(\xu_K)-f_*\le\frac{27L}{(T_{\max}+1)^2(T_{\max}-2)}.
	\end{align}
\end{theorem}
\begin{proof}
Let $1\le k_1<\ldots<k_r\le K$ be the set of iteration indices such that $\gamma_{k_i} = \bar\gamma_{k_i}$ for all $i\in [r]$. At every other iteration, $\cU_k$ is nonempty, so at least one coordinate is added to $\cI_k$. Consequently,
\begin{align}
	\label{eq:Kr_bound}
	K - r\le |\cI_K|. 
\end{align}
Moreover, before termination at most $K+1$ vectors among
$\vx_0,\ldots,\vx_K$ and at most $|\cI_K|$ column vectors of $Q$ have
been queried. 
Hence the total number of oracle calls is bounded
\begin{align}
	\label{eq:TmaxK_bound}
 T_{\max}\le K+|\cI_K|+1.
\end{align}
The reason why this is an inequality is since the count includes the inquiries at termination $K+1$, although
$\vx_{K+1}$ and $\xu_{K+1}$ are not computed. From \eqref{eq:Kr_bound} and \eqref{eq:TmaxK_bound} we then have
\begin{align}
	\label{eq:rTmax_bound}
	\ r\ge T_{\max}-2|\cI_K|-1.
\end{align}
Theorem~\ref{thm:main} result on $|\cI_K|$ and the choice of $\tau$ also yield
\begin{align}
	\label{eq:Ik_card_bound}
 |\cI_K|\le\frac{T_{\max}+1}{3}.
\end{align}
Note that if equality holds, then Theorem~\ref{thm:main} implies that
$f(\xu_K)-f_*=0$ and \eqref{eq:main_rate} is satisfied trivially. Therefore, we will assume that the inequality above is strict. Now that our assumption that $T_{\max}\ge 5$, \eqref{eq:rTmax_bound}, and \eqref{eq:Ik_card_bound} implies the trivial bound $r\ge 1$. 

For any $i\in [r]$ we have 
\begin{align}
	\label{eq:Gamma_eq}
\Gamma_{k_i} =  \Gamma_{k_i-1}(1-\gamma_{k_i})=\gamma_{k_i}^2 =\frac{4}{\left({1+\sqrt{1+4/\Gamma_{{k_i}-1}}}\right)^2}.
\end{align}
Observe that the above relation implies that $\Gamma_{k_i}\le 1$ for all $i\in [r]$. For any $k\in[K]\backslash\{k_1,\ldots, k_r\}$ we have $\Gamma_{k} = \Gamma_{k-1}(1-\gamma_{k})\le \Gamma_{k-1}$. Thus we have $\Gamma_{k_i-1}\le\Gamma_{k_{i-1}}$ for $i=2,\ldots,r$, while \eqref{eq:Gamma_eq} gives $\Gamma_{k_1}<1$. Consequently, if we define an auxiliary sequence $\{\tilde\Gamma_i\}_{i\in [r]}$ with $\tilde\Gamma_1 = 1$ and $\tilde\Gamma_i =   {4}/\left({1+\sqrt{1+4/\tilde\Gamma_{i-1}}}\right)^2$ for $i=2,\ldots,r$, then by induction and using  
the relations \eqref{eq:Gamma_eq}, $\Gamma_{k_i}\le 1$, $\Gamma_{k_i-1}\le\Gamma_{k_{i-1}}$, and the monotonicity of the last expression in \eqref{eq:Gamma_eq}, we have $\Gamma_{k_i}\le\tilde\Gamma_i$. The property of sequence $\{\tilde\Gamma_i\}_{i\in [r]}$ is well studied; e.g., by 
\cite[Lemma~4.2(e)]{wu2026universal} with $E_i=1$ we have
$
 \tilde\Gamma_r\le{4}/{(r+1)^2}.
$
Thus $\Gamma_{k_r} \le 4/(r+1)^2$ and hence
$
	\Gamma_K\le {4}/{(r+1)^2}.
$
Applying the above estimate to the result of Theorem~\ref{thm:main}, and noting \eqref{eq:rTmax_bound} 
we have
\begin{align*}
 f(\xu_K)-f_*
 \le & L\frac{6}{T_{\max}+1}\cdot\frac{4}{(r+1)^2}\left(1 - \frac{3}{T_{\max}+1}|\cI_K|\right)
 \\
 = & 
 \frac{24L(T_{\max}+1-3|\cI_K|)}
 {(T_{\max}+1)^2(r+1)^2}\le  \frac{24L(T_{\max}+1-3|\cI_K|)}
 {(T_{\max}+1)^2(T_{\max}-2|\cI_K|)^2}
 \\
 \le & \frac{27L}{(T_{\max}+1)^2(T_{\max}-2)}.
\end{align*}
Here in the last inequality we use the fact from \eqref{eq:Ik_card_bound} that
\[
 (T_{\max}-2|\cI_K|)^2
 =\frac{\left((T_{\max}-2)
 +2(T_{\max}+1-3|\cI_K|)\right)^2}{9}
 \ge\frac{8(T_{\max}-2)
 (T_{\max}+1-3|\cI_K|)}{9}.
\]
\end{proof}

From the above theorem, we can observe that the proposed algorithm achieves an $\cO(1/T^3)$ rate for solving problem \eqref{eq:problem} after $T$ oracle calls.

\section{Concluding remarks}
\label{sec:conclusion}

In this paper, we propose a first-order method with $\cO(L/T^3)$ convergence rate for solving convex quadratic programming over the unit $L_1$-ball. Our method is based on the idea of decomposing any iterate into good and bad parts and acquiring Hessian matrix column information for all the bad coordinates. It is straightforward to extend our method to problems on $L_1$-norm balls of any size. In this section, we discuss a few possible extensions of the proposed algorithm. 

First, the subproblems in Algorithm~\ref{alg:main} involve only coefficients
known after the preceding inquiries. It is certainly possible to implement an approximate subproblem solver that stops with desirable subproblem accuracy. Observing that the difficulty of the subproblems is dependent mainly on the set of inquired bad coordinates, which is independent of the problem dimension, we believe that it is reasonable to hope that the complexity of the subproblem solver is dimension independent. For simplicity in exposition, the theorem stated in this paper only concerns exact computation on known subproblems and
does not supply complexity bounds for those subproblem solvers.

The fixed level $\tau$ is also mainly our choice for simple exposition.  In
the present algorithm it is selected using the prescribed budget
$T_{\max}$, so the method is not an anytime algorithm.  One may instead
decrease the level according to the current oracle count and update
the corresponding parameters in Algorithm~\ref{alg:main}.  We believe that the same good--bad
mechanism would then suggest a single sequence that can be stopped after
any number of inquiries.  Establishing such an anytime formulation
requires additional bookkeeping because an index identified at one
level must remain consistently incorporated when the level changes;
we do not include those details here.

Finally, the observation in Section~\ref{sec:motivation} is not specific to $L_1$ in
its qualitative form.  For the unit $L_p$ ball with
$1<p<2$, \eqref{eq:l1_unique_prop} is replaced by the bounds
\(
 \|\vx^{\cI_\vx^c}\|_2^2\le\tau^{2-p}$ and $
 |\cI_\vx|\le\tau^{-p}.
\)
If the same proposed algorithmic mechanism can be carried through, we believe that it is also possible to develop a rate that interpolates between our $\cO(1/T^3)$ at
$p=1$ and the usual $\cO(1/T^2)$ at $p=2$. 

\subsection{Acknowledgement}

The author is partially supported by AFOSR grant FA9550-25-1-0278. This work was supported in part by OpenAI API credits provided by Clemson University and administered by Clemson University Research Computing and Data (RCD). 
The OpenAI ChatGPT 5.6 (Sol) model was previously used along with the Rethlas mathematical reasoning agent \cite{ju2026automated} for more than 10 hours of automated research and did not solve the problem of interest. After failed attempts, the first algorithm with rate better than $\cO(L/T^2)$ rate was developed later when the author instructed the OpenAI ChatGPT 6 (Astra) model to inspect properties associated with weak $L^1$ spaces and the good-bad part decomposition idea of harmonic analysis. The convergence rate of the first algorithm was $\cO(L\log T/T^3)$. The current version of algorithm with $\cO(L/T^3)$ rate was developed after the author instructed the Astra model to utilize the techniques appearing in a recent harmonic analysis result in \cite{ouyang2026dimension}. 
Both the Sol and Astra models were used to assist with manuscript preparation. The author have independently verified, validated, and rewritten the paper in its entirety, and takes full responsibility for the mathematical contents of the paper.

\bibliographystyle{plain}
\bibliography{article}
\end{document}